\documentclass[11pt]{article}

\usepackage[english]{babel}
\usepackage{csquotes}
\usepackage[style=alphabetic, 
            hyperref=auto,
            url=true,
            isbn=false,
            doi=false,
            block=none,
            maxnames=10,
            backend=bibtex]{biblatex}
\usepackage{amssymb}
\usepackage{amsthm}
\usepackage{amsmath}
\usepackage{mathrsfs}

\usepackage[a4paper, top=3cm, bottom=3cm]{geometry}
\usepackage[a4paper]{geometry}
\usepackage{numprint}
\usepackage[pdftex]{graphicx}
\usepackage{pgfplots}

\theoremstyle{plain}
\newtheorem{theorem}{Theorem}[section]
\newtheorem{conjecture}[theorem]{Conjecture}
\newtheorem{lemma}[theorem]{Lemma}
\newtheorem{proposition}[theorem]{Proposition}
\newtheorem{corollary}[theorem]{Corollary}
\theoremstyle{definition}

\theoremstyle{remark}
\newtheorem{remark}[theorem]{Remark}

\numberwithin{equation}{section}

\newcommand{\longto}{\longrightarrow}

\newcommand{\CC}{\mathbb{C}}
\newcommand{\ZZ}{\mathbb{Z}}
\newcommand{\FF}{\mathbb{F}}
\newcommand{\QQ}{\mathbb{Q}}
\newcommand{\NN}{\mathbb{N}}
\newcommand{\cB}{\mathcal{B}}
\newcommand{\cH}{\mathcal{H}}
\renewcommand{\SS}{\mathfrak{S}}
\renewcommand{\Im}{\operatorname{Im}}
\newcommand{\Frob}{\operatorname{Frob}}
\newcommand{\Gal}{\operatorname{Gal}}

\newcommand{\bound}{\numprint{12000}}

\newcommand{\GL}{\mathrm{GL}}
\newcommand{\SL}{\mathrm{SL}}

\title{Experimental evidence for Maeda's conjecture on modular forms
\footnote{
We thank David Harvey for asking a question that lead us to
  drastically improve our Sage implementation, and David Farmer, Gabor Wiese
and the referees for very useful comments.}
}
\author{
Alexandru Ghitza\footnote{
  Research of the first author was supported by 
a Discovery Grant from the Australian Research Council.
Some of the computations described in this paper were performed on the Sage
cluster at the University of Washington, partly supported by National 
Science Foundation Grant No. DMS-0821725, held by William Stein.}  
{} and 
Angus McAndrew\\
Department of Mathematics and Statistics\\
University of Melbourne\\
{\tt aghitza@alum.mit.edu}, {\tt mcandrew@student.unimelb.edu.au}
}
\date{}

\begin{document}
\thispagestyle{empty}

\maketitle
\begin{abstract}
  We describe a computational approach to the verification of Maeda's conjecture
  for the Hecke operator $T_2$ on the space of cusp forms of level one. We provide
  experimental evidence for all weights less than $\bound$, as well as some
  applications of these results.  The algorithm was implemented using the
  mathematical software Sage, and the code and resulting data were made
  freely available.
\end{abstract}

\section{Introduction}
\label{sect:introduction}
Modular forms come in many different types.  One of the most attractive
aspects of the theory is that, despite the apparent variety of definitions
and properties, there are some universal guiding principles (such as the
Langlands program) that serve to unify and motivate this diversity.  On the
other hand, there are some special properties that seem to occur in isolation.
One such instance is provided by a conjecture formulated by Maeda, which
indicates a behavior that seems to be specific\footnote{We must note that
  recent work of Tsaknias~\cite{Tsaknias} points to a generalisation of
Maeda's conjecture to forms of higher level and promises to shed new
conceptual light on these questions.  We thank Gabor Wiese for
bringing Tsaknias' preprint to our attention.} to modular forms of level one
on $\GL_2$.  

Before describing Maeda's conjecture in more detail, we review some basic 
definitions and properties of modular forms.  For a thorough treatment of
the background needed in this paper, the reader is invited to
consult~\cite{Stein}.

Let $k\in\ZZ$.  A \emph{modular form} of level $1$ and weight $k$ is a
holomorphic function
\begin{equation*}
  f\colon\cH\longto\CC, \quad\text{where }
  \cH=\{z\in\CC\mid \Im z>0\},
\end{equation*}
satisfying
\begin{itemize}
  \item Modularity: for all $z\in\cH$ and all
    $g=\left(\begin{smallmatrix}a&b\\c&d\end{smallmatrix}\right)\in\SL_2(\ZZ)$,
      \begin{equation*}
        f\left(\frac{az+b}{cz+d}\right)=(cz+d)^kf(z).
      \end{equation*}
  \item Holomorphicity at $i\infty$: a holomorphic function $f$ satisfying the
    modularity condition satisfies $f(z+1)=f(z)$ for all $z\in\cH$, so it has
    a Fourier expansion
    \begin{equation*}
      f(z)=\sum_{n=-\infty}^\infty a_nq^n,\quad\text{where we set }
      q=e^{2\pi i z}.
    \end{equation*}
    We ask for $f$ to be \emph{holomorphic at $i\infty$}, i.e. that $a_n=0$
    for all $n<0$.
\end{itemize}

We say that a modular form $f$ is a \emph{cusp form} if $a_0=0$.  The cusp
forms of weight $k$ form a vector space $S_k$.  These vector spaces are
equipped with a family of \emph{Hecke operators} $T_m$ (for $m\in\NN$), whose
effect on the Fourier expansion $f(q)=\sum a_nq^n$ of $f\in S_k$ is given by
\begin{equation*}
  (T_m f)(q)=\sum_{n=1}^\infty \left(\sum_{d\mid\gcd(m,n)}d^{k-1}a_{mn/d^2}\right)q^n.
\end{equation*}

The complex vector space $S_k$ has dimension 
\begin{equation*}
  d=\begin{cases}
    \left[\frac{k}{12}\right]-1 & \text{if }k\equiv 2\pmod{12},\\
    \left[\frac{k}{12}\right] & \text{if }k\not\equiv 2\pmod{12}.
  \end{cases}
\end{equation*}

Let $F$ denote the characteristic polynomial of the operator $T_2$ acting on
$S_k$, and let $d=\dim S_k$.
In the 1970s, Yoshitaka Maeda noticed that $F$ is irreducible over $\QQ$ for
all $k$ such that $d\leq 12$.  In the 1990s, Lee-Hung~\cite{LeeHung} and
Buzzard~\cite{Buzzard} studied these polynomials further and observed in a
number of cases that the Galois group of $F$ is the symmetric group
$\SS_d$.  Shortly thereafter, Maeda made the following conjectural statement:

\begin{conjecture}[Maeda~\cite{HidaMaeda}]
  Let $m>1$ and 
  let $F$ be the characteristic polynomial of the Hecke operator $T_m$ acting
  on $S_k$.  Then 
  \begin{enumerate}
    \item the polynomial $F$ is irreducible over $\QQ$;
    \item the Galois group of the splitting field of $F$ is the full symmetric
      group $\SS_d$, where $d$ is the dimension of $S_k$.
  \end{enumerate}
\end{conjecture}

The conjecture has enjoyed constant attention over the last 15 years, with
theoretical as well as computational results.  We summarize the computational
verifications in Table~\ref{tbl:known}.

\begin{table}[h]
  \begin{center}
\begin{tabular}{l|r}
  Source & weights\\ \hline
  Lee-Hung~\cite{LeeHung} & $k\leq 62$, $k\neq 60$ \\
  Buzzard~\cite{Buzzard} & $k=12\ell$, $\ell$ prime, $2\leq\ell\leq 19$ \\
  Maeda~\cite{HidaMaeda} & $k\leq 468$ \\
  Conrey-Farmer~\cite{ConreyFarmer} & $k\leq 500$, $k\equiv 0\pmod 4$  \\
  Farmer-James~\cite{FarmerJames} & $k\leq \numprint{2000}$  \\
  Buzzard-Stein, Kleinerman~\cite{Kleinerman} & $k\leq\numprint{3000}$ \\
  Chu-Wee Lim~\cite{Lim} & $k\leq\numprint{6000}$ \\
  present paper & $k\leq\bound$  \\
\end{tabular}
\end{center}
\caption{Summary of known cases of Maeda's conjecture for $T_2$}
\label{tbl:known}
\end{table}

The theoretical results focus on whether the validity of the conjecture for
a given operator $T_m$ can be used to deduce the conjecture for other
operators $T_n$.  We state three such results, each giving a partial answer to
this question.

\begin{theorem}[Conrey-Farmer-Wallace~\cite{ConreyFarmerWallace}]
  \label{thm:CFW}
  Let $k$ be a positive even integer.  Suppose there exists $n\geq 2$ such
  that the operator $T_n$ acting on $S_k$ satisfies Maeda's conjecture.  Then
  so does $T_p$ acting on $S_k$, for every prime $p$ in the set of density
  $5/6$ defined by the conditions
  \begin{equation*}
    p\not\equiv \pm 1\pmod{5}\qquad\text{or}\qquad
    p\not\equiv \pm 1\pmod{7}.
  \end{equation*}
\end{theorem}

Stated differently, this says that if Maeda's conjecture in weight $k$ holds
for one index $n$, then the density of primes for which the conjecture
fails is at most $1/6$.  The next result considers only the irreducibility
part of the conjecture, but it is stronger since it says that the density of
primes for which the conjecture fails is zero.

\begin{theorem}[Baba-Murty~\cite{BabaMurty}]
  Let $k$ be a positive even integer.  Suppose there exists a prime $p$ such
  that the characteristic polynomial of $T_p$ acting on $S_k$ is irreducible
  over $\QQ$.  Then there exists $\delta>0$ such that
  \begin{equation*}
    \#\{\ell\leq N\text{ prime}\mid \text{charpoly}(T_\ell|S_k)\text{ is reducible}\}
    \ll \frac{N}{(\log N)^{1+\delta}}.
  \end{equation*}
\end{theorem}

Finally, Ahlgren gave a simple criterion for extending the validity of
Maeda's conjecture from one index to another, and used it together with some
computer work to prove the following result.

\begin{theorem}[Ahlgren~\cite{Ahlgren}]\label{thm:Ahl}
  Let $k$ be such that $d:=\dim S_k\geq 2$.  Suppose there exists $n\geq 2$
  such that the operator $T_n$ acting on $S_k$ satisfies Maeda's conjecture.
  Then
  \begin{enumerate}
    \item $T_p$ acting on $S_k$ satisfies Maeda's conjecture for all primes
      $p\leq\numprint{4000000}$;
    \item $T_n$ acting on $S_k$ satisfies Maeda's conjecture for all
      $n\leq\numprint{10000}$.
  \end{enumerate}
\end{theorem}

We can now state our main result.

\begin{theorem}\label{thm:main}
  Let $k\leq \bound$ and let
  \begin{align*}
    n\in &\{2,\ldots,\numprint{10000}\}\cup
    \{p\text{ prime}\mid 2\leq p\leq\numprint{4000000}\}\\
    &\cup\{p\text{ prime}\mid p\not\equiv\pm 1\pmod{5}\}\cup
    \{p\text{ prime}\mid p\not\equiv\pm 1\pmod{7}\}.
  \end{align*}
  Let $F$ be the characteristic polynomial of the
  Hecke operator $T_n$ acting on the space $S_k$ of cusp forms of weight
  $k$ and level $1$.  Then $F$ is irreducible over $\QQ$ and the Galois
  group of its splitting field is the full symmetric group $\SS_d$, 
  where $d$ is the dimension of the space $S_k$.
\end{theorem}
\begin{proof}
  The statement for $T_2$ is the result of the computations described below.  
  Given this, we deduce the result for the other $T_n$ by applying the results of
  Conrey-Farmer-Wallace and Ahlgren, as stated above.
\end{proof}

Our computational approach follows the ``multimodular'' method introduced by
Buzzard in~\cite{Buzzard} and refined by Conrey-Farmer in~\cite{ConreyFarmer}.
The main improvement is the use of random primes of moderate size, instead of
going through primes consecutively until suitable ones are found.  In
Section~\ref{sect:density} we describe the theoretical foundation of this
approach, and we estimate the densities of the different types of primes we
are looking for.  This provides us with expected running times for our
randomized algorithm, the Sage implementation of which we discuss in detail in
Section~\ref{sect:implementation}.  Finally, Section~\ref{sect:applications}
gives some direct corollaries of Theorem~\ref{thm:main} to some questions
about modular forms of level one.

We have made the code and data used to verify Theorem~\ref{thm:main} available
at 

\centerline{\url{http://bitbucket.org/aghitza/maeda_data}}

\section{Polynomial factorization and Frobenius elements}
\label{sect:frobenius}
Our algorithm is based on a correspondence between the factorization of
polynomials over finite fields and the cycle decomposition of Frobenius
elements in Galois groups.  We give a short review of these
results, which go back all the way to the beginnings of algebraic number theory,
appearing for instance in the work of Frobenius.  A fascinating exposition of the
mathematics and history of these ideas is given by Stevenhagen and Lenstra
in~\cite{StevenhagenLenstra}.

We start with a bit of terminology.  If $\tau$ is a permutation on $d$ letters, 
it can be decomposed into a product of disjoint cycles, uniquely up to
permutation of the cycles.  We say that $\tau$ has \emph{cycle pattern}
$d_1^{m_1}d_2^{m_2}\ldots d_t^{m_t}$ if its decomposition contains exactly
$m_j$ cycles of length $d_j$, for $j=1,\ldots,t$.  
  (Note: $m_1d_1+m_2d_2+\ldots +m_td_t=d$.) 
If $H$ is a polynomial in
$\FF_p[X]$, we say that $H$ has \emph{factorization pattern}
$d_1^{m_1}d_2^{m_2}\ldots d_t^{m_t}$ if $H$ has exactly $m_j$ irreducible
factors of degree $d_j$ over $\FF_p$.  We recall that $H$ is said to be
\emph{separable} if it has distinct roots over $\overline{\FF}_p$.

\begin{lemma}
  \label{lem:frobenius}
  Let $F\in\ZZ[X]$ be monic, let $p$ be a prime and let $F_p\in\FF_p[X]$ be
  the reduction of $F$ modulo $p$.  If $F_p$ is separable, then there exists
  an element $\tau$ of the Galois group of $F$ such that the cycle pattern of
  $\tau$ is the same as the factorization pattern of $F_p$.
\end{lemma}

We sketch a proof of this classical result.

Fix a prime $p$ and consider the field automorphism
$\sigma\colon\overline{\FF}_p\longto\overline{\FF}_p$ given by $\sigma(a)=a^p$.
Since $\sigma$ fixes the subfield $\FF_p$, it permutes the roots of any
polynomial $H\in\FF_p[X]$.  Moreover, Galois theory tells us that the cycle
pattern of $\sigma$ (viewed as a permutation) is the same as the factorization
pattern of $H$ over $\FF_p$. 

We now take a monic polynomial $F\in\ZZ[X]$ and we let $K/\QQ$ be its
splitting field, $\mathcal{O}_K$ the ring of integers of $K$, and $G$ the
Galois group of $K/\QQ$.  Let $\mathfrak{p}$ be a prime in $\mathcal{O}_K$
over $p$.  Suppose the reduction $F_p$ of $F$ modulo $p$ is a separable
polynomial (in this case, we say that $p$ is unramified in $K/\QQ$).  Then
there is a \emph{Frobenius element} $\Frob_{\mathfrak{p}}\in G$ determined
uniquely by the property
\begin{equation*}
  \Frob_{\mathfrak{p}}(\alpha)\equiv\sigma(\alpha)\pmod{\mathfrak{p}}\qquad
  \text{for all }\alpha\in\mathcal{O}_K.
\end{equation*}

This implies that $\Frob_{\mathfrak{p}}$ permutes the roots
$\alpha_1,\ldots,\alpha_d\in\mathcal{O}_K$ of $F$ in the exact same way as
$\sigma$ permutes the roots in $\overline{\FF}_p$ of $F_p$.  We conclude that
the cycle pattern of $\Frob_{\mathfrak{p}}$ is the same as the factorization
pattern of $F_p$ over $\FF_p$.  Therefore we can take $\tau$ in the conclusion
of Lemma~\ref{lem:frobenius} to be $\Frob_{\mathfrak{p}}$.

Note that $\tau$ is not uniquely determined by $F$ and $p$, as the choice of a
prime $\mathfrak{p}$ of $\mathcal{O}_K$ above $p$ matters.  However, any two
such $\tau$ are conjugate in the Galois group.

The following result follows easily from Lemma~\ref{lem:frobenius} and the
fact that for any $F\in\ZZ[X]$ there are only finitely many primes $p$ (namely
the ones dividing the discriminant of $F$) for which $F_p$ is not separable.

\begin{theorem}[Frobenius]\label{thm:frobenius}
  Let $F\in\ZZ[X]$ be monic, let $K/\QQ$ be the splitting
  field of $F$ and let $G$ be the Galois group of $K/\QQ$.  
  Let $\deg F=m_1d_1+\ldots+m_td_t$ be a partition of $\deg F$.  
  The density of primes $p$ for which $F_p$ has factorization pattern
  $d_1^{m_1}\ldots d_t^{m_t}$ is equal to
  \begin{equation*}
    \frac{\#\{\sigma\in G\mid\text{the cycle pattern of $\sigma$ is 
    $d_1^{m_1}\ldots d_t^{m_t}$}\}
    }{\# G}.
  \end{equation*}
\end{theorem}

\section{The basic lemma and density estimates}
\label{sect:density}

Consider a monic polynomial $F\in\ZZ[X]$ of degree $d$.
Given a prime $p$, we denote by
$F_p\in\FF_p[X]$ the reduction modulo $p$ of $F$.  We say that the prime
$p$ is
\begin{enumerate}
  \item \emph{of type I} if $F_p$ is irreducible over $\FF_p$;
  \item \emph{of type II} if $F_p$ factors over $\FF_p$ into a product of
    distinct irreducible factors
    \begin{equation*}
      F_p=f_0f_1\cdots f_s
    \end{equation*}
    with
    \begin{align*}
      \deg f_0 &= 2\\
      \deg f_j &\text{ odd for }j=1,\ldots,s;
    \end{align*}
  \item \emph{of type III} if $F_p$ factors over $\FF_p$ into a product of
    distinct irreducible factors
    \begin{equation*}
      F_p=f_0f_1\cdots f_s
    \end{equation*}
    with $\deg f_0>d/2$ and prime.
\end{enumerate}

\begin{remark}\label{rmk:type4}
  Hida and Maeda use a similar approach in Section~5 of~\cite{HidaMaeda}, but 
  replace primes of type III with
  \emph{primes of type IV}, i.e. $p$ such that $F_p=f_0f_1$ with $f_0$, $f_1$
  distinct and irreducible, and $\deg f_0=1$.  We will see below that primes
  of type III are significantly more common (and therefore better suited for
  our algorithm) than those of type IV.
\end{remark}

\begin{remark}
These types are not necessarily mutually exclusive: if $d$ itself is
prime, then a prime $p$ of type I is clearly also of type III.
\end{remark}

\begin{remark}
In either of the three types, the conditions imply that the reduced polynomial
$F_p$ is separable:
\begin{itemize}
  \item If $p$ is a prime of type I, then $F_p$ is irreducible, hence
    separable.
  \item If $p$ is a prime of type II or III, $F_p$ is a product of distinct
    irreducible factors.  Each of the factors is then separable, and they 
    cannot have any common roots, since otherwise they would have a
    nonconstant common factor and would therefore be reducible.  Hence $F_p$
    has distinct roots.
\end{itemize}
\end{remark}

Our computational approach to Maeda's conjecture is based on the following
result, first proved in a special case in~\cite{Buzzard} and then generalized
in~\cite{ConreyFarmer}.

\begin{lemma}[Buzzard, Conrey-Farmer]
  Let $F\in\ZZ[X]$ be a monic polynomial of degree $d$.  Suppose that $F$ has
  primes of respective types I, II and III.  Then $F$
  is irreducible over $\QQ$ and its splitting field over $\QQ$ has full Galois
  group $\SS_d$.
\end{lemma}
\begin{proof}
  The fact that $F$ is irreducible is immediate from the existence of a prime of type
  I. 

  Let $K/\QQ$ be the splitting field of $F$ and let $G$ be the Galois group of
  $K/\QQ$.  Since $F$ is irreducible, $G$ is a transitive subgroup of $\SS_d$. 

  We also have a prime of type II.
  By Lemma~\ref{lem:frobenius}, there exists $\tau_1\in G$ whose decomposition
  into disjoint cycles contains exactly one even cycle (of length
  $2$).  Let $a$ be the least common multiple of the lengths of the other
  cycles in $\tau_1$, then $\tau_1^a\in G$ is a transposition.
  
  Finally, there is a prime of type III.  By Lemma~\ref{lem:frobenius}, there
  exists $\tau_2\in G$ whose decomposition into disjoint cycles contains one
  cycle of prime length $p>d/2$.  Therefore the other cycles have
  lengths that are coprime to $p$; letting $b$ denote the least common
  multiple of these lengths, we find that $\tau_2^b\in G$ is a $p$-cycle.

  We now use the existence of these elements of $G$ to conclude that
  $G=\SS_d$.  For $i,j\in S=\{1,\ldots,d\}$, write $i\sim j$ if $i=j$ or if
  the transposition $(i~j)$ is in $G$.  This is an equivalence relation on
  $S$.  Since $G$ is
  transitive, each equivalence class has the same number $n$ of elements 
  and it follows that $n\mid d=\#S$. Note that $n>1$ since $G$ contains at least one
  transposition, namely $\tau_1^a$. 
  Let $T$ be the subset of $S$ permuted by
  $\tau_2^b$, and let $G_T$ be the subgroup of $G$ fixing $S\setminus T$.
  Define an equivalence relation on $T$ by $i\simeq j$ if $i=j$ or if the transposition
  $(i~j) \in G_T$. As before, each equivalence class has the same number $m$ of elements
  and $m\mid p=\#T$. Since $n>1$, we have $m>1$, so $m=p$ since $p$ is prime.
  But $n\geq m$ because $G_T\subset G$. Thus $n>d/2$, so $n=d$. This implies $G=\SS_d$. 
\end{proof}

Our algorithm will consist of picking random primes and checking whether they
are of type I, II or III for the characteristic polynomial of the Hecke
operator $T_2$.  According to Theorem~\ref{thm:frobenius}, 
it is therefore important to estimate the number of permutations having
certain types of cycle patterns.  For a fixed pattern, the following
well-known result 
(see, for instance, Proposition 1.3.2 of~\cite{Stanley1}) gives an exact
expression for the number of permutations.

\begin{lemma}\label{lem:cycletype}
  Let an element $\sigma$ of $\SS_d$ have cycle pattern $d_1^{m_1}d_2^{m_2}\ldots d_t^{m_t}$, 
  where $m_i$ is the number of times a cycle of length $d_i$ appears in the cycle
  decomposition of $\sigma$. 
  The number
   of elements of $\SS_d$ of cycle pattern $d_1^{m_1}d_2^{m_2}\ldots d_t^{m_t}$ is equal to
  \begin{equation*}
    C(d_1^{m_1}d_2^{m_2}\ldots d_t^{m_t})=\frac{d!}{\prod_{j=1}^t\left(d_j^{m_j}m_j!\right)}.
  \end{equation*}
\end{lemma}

\begin{proposition}\label{prop:type1}
  The density of primes of type I is 
  \begin{equation*}
    D_I(d)=\frac{1}{d}.
  \end{equation*}
\end{proposition}
\begin{proof}
  Primes of type I correspond to $d$-cycles in $\SS_d$.  Each such cycle can
  be written uniquely as a sequence $1,a_1,\ldots,a_{d-1}$, where
  $a_1,\ldots,a_{d-1}\in\{2,\ldots,d\}$ can appear in any order.  Therefore
  there are $(d-1)!$ $d$-cycles, and by Theorem~\ref{thm:frobenius},
  the density of primes of type I is
  \begin{equation*}
    \frac{(d-1)!}{d!}=\frac{1}{d}.
  \end{equation*}
\end{proof}

In order to state our result on primes of type II, recall that
for $n\in\ZZ_{>0}$ odd, the \emph{double factorial} $n!!$ of $n$ is
the product of all the odd positive integers less than or equal to $n$.

\begin{proposition}\label{prop:type2}
  Let $d>2$ and let $\tilde{d}$ be the largest even integer such that 
  $\tilde{d}\leq d$.
  The density of primes of type II is given by
  \begin{equation*}
    D_{II}(d)=\frac{[(\tilde{d}-3)!!]^2}{2(\tilde{d}-2)!}
  \end{equation*}
  and satisfies the inequality
  \begin{equation*}
    D_{II}(d)>\frac{1}{4\sqrt{d}}.
  \end{equation*}
\end{proposition}
\begin{proof}
  Primes of type II correspond to elements in $\SS_d$ containing a $2$-cycle and no other
  even cycles. There are $\binom{d}{2}$ $2$-cycles in $\SS_d$; fixing a $2$-cycle, we 
  need the number $\mathscr{O}(d-2)$ of elements of odd order in $\SS_{d-2}$.
  We have
  \begin{equation*}
    D_{II}(d)=\frac{1}{d!}\binom{d}{2}\mathscr{O}(d-2)
    =\frac{\mathscr{O}(d-2)}{2(d-2)!}.
  \end{equation*}
  The sequence $(\mathscr{O}(n)\mid n\in\NN)$ appears in nature in several
  guises, see~\cite{OEIS}.  The recurrence formulas that appear there and in
  Chapter IV of~\cite{Riordan} easily give the exact expression
  \begin{equation*}
    \mathscr{O}(n)=\begin{cases}
        (n-1)!!,&\text{if }n\text{ is even}\\
        (n-2)!!n,&\text{if }n\text{ is odd,}
      \end{cases}
  \end{equation*}
  which immediately provides us with the 
  exact expression for $D_{II}(d)$ in the statement.

  It remains to establish the lower bound.  Write $\tilde{d}=2c$ for some
  $c\in\ZZ$ (recall that $\tilde{d}$ is the largest even integer less than or
  equal to $d$).
  Then
  \begin{equation}\label{eq:oh}
    \mathscr{O}(d-2)=\frac{[(2c-3)!!]^2}{2(2c-2)!}=\frac{(2c-3)!}{2^{2c-3}(2c-2)[(c-2)!]^2}.
  \end{equation}
  We use the following bounds on the factorial, which can be thought of as an
  effective version of Stirling's approximation and were obtained by Robbins
  (see~\cite{Robbins} and Section~II.9 in~\cite{Feller}):
  \begin{equation*}
    \sqrt{2\pi}n^{n+\frac{1}{2}}e^{-n+\frac{1}{12n+1}}<n!<\sqrt{2\pi}n^{n+\frac{1}{2}}e^{-n+\frac{1}{12n}}.
  \end{equation*}
  Then by using the lower bound for the numerator and the upper bound for the
  denominator on the right hand side of Equation~\eqref{eq:oh},
  we obtain
  \begin{equation*}
    D_{II}(d)>\frac{2c-3}{2c-2}\,\frac{1}{2\sqrt{\pi(c-2)}}\,
    e^{\frac{1}{24(c-2)+1}-\frac{1}{6(c-2)}}
    >\left(\frac{9}{10}\right)^2\frac{\sqrt{2}}{\sqrt{\pi}}\frac{1}{\sqrt{d}},
  \end{equation*}
  where we use the elementary inequalities (valid for $c>6$):
  \begin{align*}
    e^{\frac{1}{24(c-2)+1}-\frac{1}{6(c-2)}}>e^{-\frac{5}{24c}}&>\frac{9}{10},\\
    \frac{2c-3}{2c-2}&\geq \frac{9}{10}.
  \end{align*}
  Since
  \begin{equation*}
    \left(\frac{9}{10}\right)^2\frac{\sqrt{2}}{\sqrt{\pi}}>\frac{1}{4},
  \end{equation*} 
  this gives us the desired lower bound for $d>12$, and the remaining
  cases $2<d\leq 12$ are easily checked.
\end{proof}

\begin{proposition}
  The density of primes of type III is
  \begin{equation*}
    D_{III}(d)=\sum_{d/2<\ell\leq d, \,\ell\text{ prime}} \frac{1}{\ell}.
  \end{equation*}
  If $d>2$, then
  \begin{equation*}
    D_{III}(d)>\frac{1}{d}.
  \end{equation*}
\end{proposition}
\begin{proof}
  Fix a prime $\ell$ such that $d/2<\ell\leq d$.  According to
  Theorem~\ref{thm:frobenius}, we need to count the number of elements of
  $\SS_d$ that contain an $\ell$-cycle.  Choosing the $\ell$-cycle itself
  involves the $\binom{d}{\ell}$ ways of picking its constituents, which can
  then be rearranged within the cycle in $(\ell-1)!$ ways.  It remains to
  take into account the number of permutations of the remaining $d-\ell$ symbols,
  so overall we have
  \begin{equation*}
    \binom{d}{\ell}(\ell-1)!(d-\ell)!=\frac{d!}{\ell}
  \end{equation*}
  elements of $\SS_d$ containing an $\ell$-cycle, which gives the stated
  density.

  The inequality given in the statement follows from Bertrand's postulate
  (proved by Chebyshev), which says that for any integer $n>1$ there is at
  least one prime $\ell$ such that $n<\ell<2n$.
\end{proof}

We can get a much better lower bound on the density $D_{III}$ by using some
recent results of Dusart on explicit estimates for sums over primes.

\begin{theorem}[Dusart, Theorem 6.10 in~\cite{Dusart}]
  Let $B\approx 0.26149$ denote the Meissel-Mertens constant.  For
  all $x>1$ we have
  \begin{equation}\label{eq:rec_lower}
    \log\log x+B-\left(\frac{1}{10\log^2 x}+\frac{4}{15\log^3 x}\right)\leq
    \sum_{p\leq x}\frac{1}{p}.
  \end{equation}
\end{theorem}

We will also need an upper bound on the sum of the reciprocals of primes up to
$x$, but Dusart's upper bound only holds for $x\geq\numprint{10372}$.  For our
purposes, the following weaker result is sufficient: for all $x>1$ we have
\begin{equation}\label{eq:rec_upper}
  \sum_{p\leq x}\frac{1}{p}\leq\log\log x + B +\frac{1}{\log^2 x}.
\end{equation}
(This inequality can be found in Theorem 8.8.5 of~\cite{BachShallit}.)

\begin{proposition}\label{prop:type3}
  If $d>10$, then
  \begin{equation*}
    D_{III}(d)>\frac{1}{3\log d}.
  \end{equation*}
\end{proposition}
\begin{proof}
  We put together inequalities~\eqref{eq:rec_lower} and~\eqref{eq:rec_upper}
  to get
  \begin{equation*}
    D_{III}(d)>\log\log d-\log\log\frac{d}{2}-\frac{1}{10\log^2 d}
    -\frac{4}{15\log^3 d}-\frac{1}{\log^2\frac{d}{2}}.
  \end{equation*}
  We write
  \begin{equation*}
    \log\log d - \log\log\frac{d}{2} =\log\left(1+\frac{\log 2}{\log d-\log
    2}\right)
  \end{equation*}
  and use the inequality
  \begin{equation*}
    \log(1+x)\geq x-\frac{x^2}{2}+\frac{x^3}{3}-\frac{x^4}{4}
    \qquad\text{for all }1<x\leq 1
  \end{equation*}
  to get that for all $d\geq 4$
  \begin{align*}
    D_{III}(d)>&\frac{1}{\log d-\log 2}\left[
      \log 2
      -\left(\frac{\log^2 2}{2}+\frac{11}{10}\right)\frac{1}{\log d-\log
      2}\right.\\
      &\left.\phantom{\frac{1}{\log d-\log 2}} 
      -\left(\frac{4}{15}-\frac{\log^3 2}{3}\right)\frac{1}{(\log d-\log 2)^2}
      -\frac{\log^4 2}{4}\frac{1}{(\log d-\log 2)^3}
      \right]\\
      >&\frac{1}{\log d}\left[
      0.693-1.341\frac{1}{\log d-\log 2}
      -0.156\frac{1}{(\log d-\log 2)^2}\right.\\
      &\left.\phantom{\frac{1}{\log d}}
      -0.058\frac{1}{(\log d-\log 2)^3}
      \right].
  \end{align*}
  If $d>94$, then the expression in the brackets is bigger than $1/3$, and we
  get the desired inequality.  We check that it holds for the remaining cases
  $10<d\leq 94$ by computation.
\end{proof}

For completeness, we treat the case of primes of type IV, as defined in
Remark~\ref{rmk:type4}.

\begin{proposition}\label{prop:type4}
  Let $d>1$.  The density of primes of type IV is
  \begin{equation*}
    D_{IV}(d)=\frac{1}{d-1}.
  \end{equation*}
\end{proposition}
\begin{proof}
  We need to count the number of $(d-1)$-cycles in $\SS_d$.  There are $d$
  choices for the letter that is fixed, and $(d-2)!$ choices for permuting the
  other letters appropriately, therefore the density of primes of type IV is
  \begin{equation*}
    \frac{d(d-2)!}{d!}=\frac{1}{d-1}.
  \end{equation*}
\end{proof}

\section{Implementation and results}
\label{sect:implementation}
Our approach is a randomized version of the algorithm
from~\cite{ConreyFarmer}, based on the results introduced in the previous
section.  We implemented this algorithm using the mathematical software Sage, 
see~\cite{Sage}.

Here is a description of the main steps used to verify Maeda's conjecture for
a fixed weight $k$; in those cases where a major step is delegated to a
component of Sage (rather than using native Sage code), we mention the
relevant component.
\begin{enumerate}
  \item\label{itm:vmbasis} Compute the Victor Miller basis $\cB$ for $S_k$ up 
    to precision
    $2(d+2)$, where $d$ is the dimension of $S_k$.  The Sage implementation of
    this basis uses~\cite{FLINT} polynomials as the internal data structure.
  \item\label{itm:hecke} Compute the matrix $M$ of the Hecke operator $T_2$ 
    with respect to the
    basis $\cB$ -- this is very efficient since the basis $\cB$ is
    echelonized.
  \item\label{itm:random} Pick a random prime $p<2^{20}$, uniformly over this
    range.  (This choice of upper
    bound gives a large enough range so that it is likely to contain primes of
    type we are looking for, but not so large that the arithmetic over $\FF_p$
    gets too expensive.)
  \item Reduce $M$ modulo $p$ and compute the characteristic polynomial
    $F_p\in \FF_p[X]$.  The characteristic polynomial is computed by
    the~\cite{Linbox} library.
  \item Is $F_p$ irreducible?  If so, $p$ is a prime of type I.  The
    irreducibility test uses~\cite{FLINT}.
  \item Factor $F_p$ over $\FF_p$ and use this factorization to decide whether
    $p$ is a prime of type II or III.  The factorization is done
    by~\cite{FLINT}.
  \item Repeat from step~(\ref{itm:random}) until we have found at
    least one prime of each type.
\end{enumerate}

According to Propositions~\ref{prop:type1}, \ref{prop:type2} and
\ref{prop:type3}, we expect to look on average at $d$ primes before we find
one of type I, at $4\sqrt{d}$ primes to find one of type II, and at $3\log d$
primes to find one of type III.

The actual performance of this algorithm (as well as a comparison to the
consecutive version of the algorithm, used in~\cite{ConreyFarmer}) is
illustrated in Figure~\ref{fig:histogram}.  Some care needs to be taken in
interpreting the graphs:
\begin{itemize}
  \item There is no difference in running times for Steps~(\ref{itm:vmbasis}) 
    and~(\ref{itm:hecke}), which 
    are common between the two algorithms.
  \item As the weight increases, the major component of the running time is
    finding a prime of type I.  Therefore, even though the randomized
    algorithm does much better at finding primes of types II and III, this
    advantage has only a minor impact on the overall running time.
  \item In the range illustrated in the graphs (i.e. weights less than
    $\numprint{2000}$), the randomized algorithm required on average one third
    of the number of primes needed by the consecutive algorithm.  However,
    some of this is counteracted by the fact that the consecutive algorithm
    works with much smaller primes, which are faster to test.
  \item Overall, for weights less than $\numprint{2000}$, the randomized 
    algorithm was about twice as fast as the consecutive one.  
\end{itemize}

It would be very interesting to understand why small primes are ill-suited for
the purposes of this multimodular algorithm.  We can only offer a heuristic
reason: we observed that the discriminant of the Hecke operator $T_2$ tends to
be highly divisible by a lot of small primes; this means that the
characteristic polynomial of $T_2$ is not squarefree at these primes, which
disqualifies them from being primes of types I, II or III.

For weight $k=\numprint{12000}$, the entire verification took about 46 hours.
The majority (93\%) of the time was spent looking for a prime of type I; this
required testing \numprint{1783} primes, and each test took about 87 seconds.
The computation of the Victor Miller basis took about 2.2 hours, and the
computation of the characteristic polynomial of $T_2$ took about 1 hour.

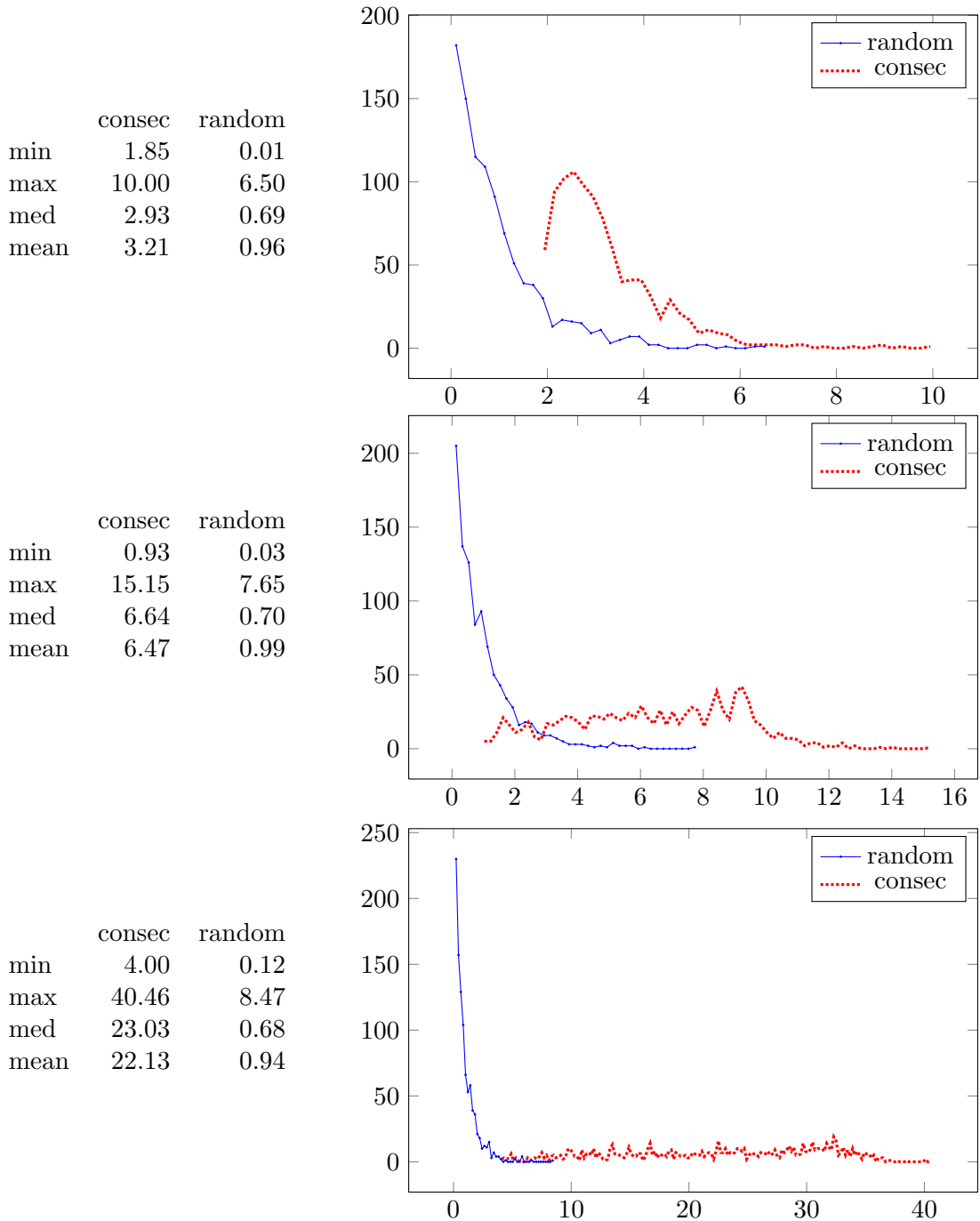
\begin{figure}[h]
  \begin{center}
  \raisebox{3.3cm}{
  \begin{tabular}{lrr}
    & consec & random\\
    min & $1.85$ & $0.01$\\
    max & $10.00$ & $6.50$\\
    med & $2.93$ & $0.69$\\
    mean & $3.21$ & $0.96$
  \end{tabular}}\qquad
  \begin{tikzpicture}
    \begin{axis}[
        height=7cm,
        width=10cm,
        mark size=0.3pt,
    ]
    \addplot file {rh1};
    \addlegendentry{random}

    \addplot[densely dotted, red, very thick] file {ch1};
    \addlegendentry{consec}
    \end{axis}
  \end{tikzpicture}

  \raisebox{3.3cm}{
  \begin{tabular}{lrr}
    & consec & random\\
    min & $0.93$ & $0.03$\\
    max & $15.15$ & $7.65$\\
    med & $6.64$ & $0.70$\\
    mean & $6.47$ & $0.99$
  \end{tabular}}\qquad
  \begin{tikzpicture}
    \begin{axis}[
        height=7cm,
        width=10cm,
        mark size=0.3pt,
    ]
    \addplot file {rh2};
    \addlegendentry{random}

    \addplot[densely dotted, red, very thick] file {ch2};
    \addlegendentry{consec}
    \end{axis}
  \end{tikzpicture}

  \raisebox{3.3cm}{
  \begin{tabular}{lrr}
    & consec & random\\
    min & $4.00$ & $0.12$\\
    max & $40.46$ & $8.47$\\
    med & $23.03$ & $0.68$\\
    mean & $22.13$ & $0.94$
  \end{tabular}}\qquad
  \begin{tikzpicture}
    \begin{axis}[
        height=7cm,
        width=10cm,
        mark size=0.3pt,
    ]
    \addplot file {rh3};
    \addlegendentry{random}

    \addplot[densely dotted, red, very thick] file {ch3};
    \addlegendentry{consec}
    \end{axis}
  \end{tikzpicture}
\end{center}
\caption{Histograms illustrating the number of primes tested before finding a
  prime of type I, II, respectively III, in weights up to $\numprint{2000}$.  
  In each graph, the numbers on the
$x$-axis represent the ratio $N/E$ of the actual number of primes tested over
the expected number of primes (coming from the densities described in
Section~\ref{sect:density}).  The $y$-value represents the number of weights
featuring (a small neighborhood of) that particular ratio $N/E$.  The blue 
continuous line corresponds to our randomized algorithm, while the red dotted 
line corresponds to the consecutive algorithm from~\cite{ConreyFarmer}.  
As an example: in the top graph, the global maximum on the continuous line is
at $(0.1, 182)$, meaning that for $182$ weights, the number of candidates for
a prime of type I tested in the randomized algorithm was about $1/10$ of the 
expected number of primes.}
\label{fig:histogram}
\end{figure}

\section{Some applications}
\label{sect:applications}

We record some immediate consequences of Theorem~\ref{thm:main}.

\subsection{Non-vanishing of $L$-functions}
A modular form is called an \emph{eigenform} if it is an eigenvector for all
the Hecke operators $T_n$.
The \emph{$L$-function} associated to an eigenform 
$f=\sum_{n=1}^\infty a_n q^n$ of weight $k$ is given by
\begin{equation*}
  L(f, s)=\sum_{n=1}^\infty \frac{a_n}{n^s}.
\end{equation*}
If $k\equiv 2\pmod{4}$, the functional equation of $L$ implies that
$L(f, k/2)=0$.  It is believed that if $k\equiv 0\pmod{4}$, then $L(f, k/2)\neq
0$.  The following result follows immediately from work of Conrey-Farmer: 

\begin{corollary}[see Theorem 1 in~\cite{ConreyFarmer}]
  Suppose $k\equiv 0\pmod{4}$ and $k\leq\bound$.  Then
  $L(f, k/2)\neq 0$ for any cuspidal eigenform $f$ of level $1$ and weight $k$.
\end{corollary}

\subsection{Base change for totally real fields}
It is in the context of this work of Hida and Maeda that
Maeda's conjecture was formulated.  We content ourselves with giving a general
description of this application, and we refer the interested reader
to~\cite{HidaMaeda} for details. 

Let $f\in S_k$ be a Hecke eigenform.  For each prime $p$, there is a $p$-adic
Galois representation
\begin{equation*}
  \rho\colon\Gal\left(\overline{\QQ}/\QQ\right)
  \longto\GL_2\left(\overline{\QQ}_p\right).
\end{equation*}
There is an Artin $L$-function $L(\rho, s)$ attached to $\rho$, and the
relation between $\rho$ and $f$ can be summarized by
\begin{equation*}
  L(\rho, s) = L(f, s).
\end{equation*}

Now let $E$ be a number field.  There is a purely algebraic notion of a
cohomological eigenform $\hat{f}$ on $\GL_2(\mathbb{A}_E)$, where 
$\mathbb{A}_E$ is the ring
of adeles of $E$.  We say that $\hat{f}$ is a \emph{base change of $f$ to $E$}
if
\begin{equation*}
  L(\hat{f}, s) = L(\rho_E, s),
\end{equation*}
where $\rho_E\colon\Gal(\overline{\QQ}/E)\longto\GL_2(\overline{\QQ}_p)$ is
the restriction of $\rho$ to $E$.

The work of Hida and Maeda, together with Theorem~\ref{thm:main}, implies that
for $k\leq\bound$ and a totally real field $E$ satisfying some ramification
conditions, any eigenform $f\in S_k$ has a base change to $E$.

\subsection{Eigenforms divisible by eigenforms}
It is easy to see from the definition of a modular form that if $f_1$ and
$f_2$ are modular forms of respective weights $k_1$ and $k_2$, then the
product $f_1f_2$ is a modular form of weight $k_1+k_2$.  In other words,
modular forms of all weights put together form a graded algebra
\begin{equation*}
  M=\bigoplus_{k\in\ZZ} M_k.
\end{equation*}

A natural question is whether the product of eigenforms can be an eigenform.
This will clearly happen for small weights (for instance, when the product
lives in a one-dimensional space of cusp forms).  Since the Hecke operators
do not act on the entire algebra $M$ of modular forms (they act
differently on the graded pieces $M_k$), it seems reasonable that the
one-dimensional coincidences are the only situation in which a product of
eigenforms is an eigenform.  Such questions have been studied by several
authors, with the latest results appearing in a recent paper by
Beyerl-James-Xue~\cite{BeyerlJamesXue}.  They consider the more general question
of divisibility of an
eigenform by another eigenform, i.e. relations of the form $h=fg$ where
$f,g,h$ are modular forms and $f,h$ are eigenforms.  The relation with Maeda's 
conjecture is discussed in Section 6 of~\cite{BeyerlJamesXue}, and 
Theorem~\ref{thm:main} implies the following result.

\begin{corollary}
  \label{cor:div_eigen}
  Let $h$ be a cuspidal eigenform of weight $\leq\bound$, and let 
  $f$ be an eigenform
  (which could be cuspidal or Eisenstein).  Then $h=fg$ for some modular form
  $g\in M_k$ with $k>2$ if and only if we are in one of the cases listed in 
  Table~\ref{tbl:div_eigen}.
\end{corollary}

\begin{table}[h]
  \begin{center}
  \begin{tabular}{r|r|l}
    weight of $f$ & weight of $g$ modulo $12$ & nature of $f$ \\ \hline
    $4$ & $0,4,6,10$ & Eisenstein \\
    $6$ & $0,4,8$ & Eisenstein \\
    $8$ & $0,6$ & Eisenstein \\
    $10$ & $0,4$ & Eisenstein \\
    $12$ & $0,2,4,6,8,10$ & cuspidal\\
    $14$ & $0$ & Eisenstein \\
    $16$ & $0,4,6,10$ & cuspidal\\
    $18$ & $0,4,8$ & cuspidal\\
    $20$ & $0,6$ & cuspidal\\
    $22$ & $0,4$ & cuspidal\\
    $26$ & $0$ & cuspidal
  \end{tabular}
\end{center}
  \caption{The only cases in which a cuspidal eigenform of weight $\leq\bound$
  can be factored into $h=fg$ with $f$ an eigenform, see
  Corollary~\ref{cor:div_eigen}.}
  \label{tbl:div_eigen}
\end{table}

\subsection{Distinguishing Hecke eigenforms}

How many initial Fourier coefficients are necessary to completely determine a
Hecke eigenform?  Theorem 1 in~\cite{Ghitza} says that $a_2$, $a_3$ and $a_4$
are sufficient, but our computational verification of Maeda's conjecture gives
a stronger result\footnote{Corollary~\ref{cor:distinguish} relies on
Theorem~\ref{thm:main}, which holds for $k\leq\bound$, but also needs another
computation from~\cite{Ghitza}, which has only been done for weights up to
$\numprint{10000}$.}:

\begin{corollary}[see Theorem 6 in~\cite{Ghitza}]
  \label{cor:distinguish}
  Let $f$ and $g$ be cuspidal eigenforms of level $1$ and (possibly distinct)
  weights
  $\leq\numprint{10000}$.  Then $a_2(f)=a_2(g)$ if and only if $f=g$.
\end{corollary}

\printbibliography

\end{document}